\documentclass[11pt]{article}

\usepackage[margin=1in]{geometry}
\usepackage{amsmath,amssymb,amsthm,mathtools}
\usepackage{bm}
\usepackage{enumitem}
\usepackage{hyperref}
\usepackage{mathrsfs}
\usepackage[utf8]{inputenc}

\newtheorem{proposition}{Proposition}

\newtheorem{theorem}{Theorem}

\newtheorem{definition}{Definition}
\newtheorem{remark}{Remark}

\hypersetup{
  colorlinks=true,
  linkcolor=blue,
  citecolor=blue,
  urlcolor=blue
}

\title{Observable Geometry of Singular Statistical Models}
\author{Sean Plummer\\ snplmmr@gmail.com}
\date{\today}

\begin{document}

\maketitle

\begin{abstract}
Singular statistical models arise whenever different parameter values induce the same distribution, leading to non-identifiability and a breakdown of classical asymptotic theory. While existing approaches analyze these phenomena in parameter space, the resulting descriptions depend heavily on parameterization and obscure the intrinsic statistical structure of the model. In this paper, we introduce an invariant framework based on \emph{observable charts}: collections of functionals of the data distribution that distinguish probability measures. These charts define local coordinate systems directly on the model space, independent of parameterization. We formalize \emph{observable completeness} as the ability of such charts to detect identifiable directions, and introduce \emph{observable order} to quantify higher-order distinguishability along analytic perturbations. Our main result establishes that, under mild regularity conditions, observable order provides a lower bound on the rate at which Kullback–Leibler divergence vanishes along analytic paths. This connects intrinsic geometric structure in model space to statistical distinguishability and recovers classical behavior in regular models while extending naturally to singular settings. We illustrate the framework in reduced-rank regression and Gaussian mixture models, where observable coordinates reveal both identifiable structure and singular degeneracies. These results suggest that observable charts provide a unified and parameterization-invariant language for studying singular models and offer a pathway toward intrinsic formulations of invariants such as learning coefficients.

\end{abstract}

\section{Introduction}

Classical statistical theory is built on local approximations of parametric models \cite{van2000asymptotic}. Under regularity conditions, a statistical model can be treated as a smooth manifold, with geometry determined by the score function and Fisher information. In this setting, asymptotic behavior is governed by first-order properties: local identifiability, quadratic expansion of the log-likelihood, and nondegenerate Fisher information. However, many modern statistical models are \emph{singular}: distinct parameter values may correspond to the same probability distribution, leading to non-identifiability and degenerate local geometry. This occurs in a wide range of settings, including mixture models, neural networks, and latent variable models. In such cases, classical asymptotic theory based on regular parameterizations breaks down, and key quantities such as Fisher information fail to capture the true structure of the model. Existing approaches to singular models, notably singular learning theory, analyze these phenomena by studying the geometry of parameter space, often after resolution of singularities \cite{watanabe2009algebraic}. While powerful, these methods remain tied to particular parameterizations, and the resulting descriptions can obscure the intrinsic statistical structure, which depends only on the set of distributions realized by the model. This suggests that the intrinsic object of interest is the model space $M$, not the parameter space $\Theta$. 

In this work, we propose a complementary perspective that operates directly on the \emph{model space} of probability distributions. Our central idea is to describe local structure using collections of \emph{observables}: functionals of the distribution that distinguish nearby models. Such collections define coordinate systems on model space that are invariant to reparameterization and reflect only statistically meaningful variation. We formalize this idea through the notion of an \emph{observable chart}, and introduce \emph{observable completeness} to characterize when such charts capture all identifiable directions. To study singular behavior, we further define \emph{observable order}, which measures the rate at which observables change along analytic perturbations. Observable order separates identifiable and non-identifiable directions in a way that is intrinsic to the model space and independent of parameterization. This provides a hierarchy of distinguishability that extends beyond first-order (Fisher) geometry. 

Our main theoretical result shows that observable order controls statistical distinguishability: under mild regularity conditions, it provides a lower bound on the rate at which Kullback–Leibler divergence vanishes along analytic paths. This establishes a direct connection between intrinsic geometric structure in model space and asymptotic statistical behavior. We illustrate the framework in reduced-rank regression and Gaussian mixture models, where observable coordinates recover identifiable structure and reveal how singular directions manifest through higher-order effects. These examples demonstrate that observable charts can separate intrinsic statistical features from parameterization artifacts. More broadly, our results suggest that singular statistical phenomena may be understood through intrinsic geometric structures defined on model space, rather than through particular parameterizations. In particular, observable order along analytic paths parallels valuation-based descriptions in singular learning theory, pointing toward a potential reformulation of learning coefficients in invariant terms.

\subsection{Related work}

This work connects to several strands of statistical theory, including classical information geometry, singular learning theory, and the study of identifiable structure in statistical models. In regular parametric models, local statistical behavior is characterized by the score function and Fisher information, leading to a Riemannian geometric structure on the model manifold \cite{amari2000methods, van2000asymptotic}. Our first-order observable framework recovers this classical geometry by expressing tangent directions through expectation functionals, providing an alternative coordinate-free perspective on Fisher geometry. Singular statistical models, including mixture models, latent-variable models, and neural networks, have been extensively studied in singular learning theory (SLT), where asymptotic behavior is governed by algebraic-geometric invariants such as the real log canonical threshold (RLCT) \cite{watanabe2009algebraic}. In this framework, the local behavior of the Kullback–Leibler divergence near singularities determines learning dynamics and generalization error \cite{watanabe2010asymptotic}. Unlike singular learning theory, which analyzes parameter space after resolution of singularities, our approach operates directly on model space and requires no reference to a particular parameterization. The resulting structure is therefore intrinsic and applies uniformly across equivalent representations of the same statistical model. Our observable charts extend this perspective by systematically organizing expectation functionals into local coordinate systems that capture higher-order structure. Our notion of observable order provides a complementary, functionally defined description of local model structure, suggesting a potential intrinsic interpretation of RLCT-type invariants in terms of observable expansions. There is also a long tradition of studying statistical models through moment and expectation-based representations. In exponential families, sufficient statistics provide global coordinate systems via expectation parameters \cite{brown1986fundamentals}. More generally, moment-based methods and cumulant expansions have been used to study identifiability and structure in mixture models and latent-variable models \cite{allman2009identifiability, anandkumar2014tensor}.  Finally, our approach is related to work on identifiability and equivalence in statistical models, where the goal is to characterize the model image independently of parameterization \cite{rothenberg1971identification}. By focusing on observable functionals, we explicitly construct finite-dimensional representations of the model image and analyze their local geometry, separating intrinsic statistical structure from parameterization artifacts. Overall, the observable framework developed here can be viewed as bridging classical differential-geometric approaches to statistics with algebraic and analytic perspectives on singular models, while providing a constructive and functionally grounded description of local model structure.

\section{Observable Representations of Statistical Models}

\subsection{Model as a subset of distributions}

Let $\mathcal P(\mathcal X)$ denote a space of probability distributions on a measurable space $\mathcal X$. A parametric statistical model is typically specified via a map
\[
\Phi : \Theta \to \mathcal P(\mathcal X), \qquad \theta \mapsto P_\theta,
\]
where $\Theta \subset \mathbb R^d$ is a parameter space. We define the \emph{model image}
\[
M = \Phi(\Theta) \subset \mathcal P(\mathcal X).
\]
The map $\Phi$ need not be injective; distinct parameter values may correspond to the same distribution. As a result, the geometry of $\Theta$ may contain redundancies that are not intrinsic to the statistical model. Throughout this paper, we treat $M$ as the primary object of interest.

\subsection{Observables and Observable Charts}

Let $\mathcal F$ be a class of measurable functions $f : \mathcal X \to \mathbb R$ such that expectations $\mathbb E_P[f]$ are well-defined for all $P \in M$. Each $f \in \mathcal F$ defines an \emph{observable functional}
\[
\psi_f : M \to \mathbb R, \qquad \psi_f(P) = \mathbb E_P[f].
\]
Finite collections of observables define maps into Euclidean space $\mathbb{R}$. Let $f_1, \dots, f_m \in \mathcal F$. Define
\[
\Psi : M \to \mathbb R^m, \qquad
\Psi(P) = \bigl(\mathbb E_P[f_1], \dots, \mathbb E_P[f_m]\bigr).
\]
We refer to $\Psi$ as an \emph{observable chart}. The image
\[
O = \Psi(M) \subset \mathbb R^m
\]
provides a finite-dimensional presentation of the model. Composing with the parameter map yields
\[
\Psi \circ \Phi : \Theta \to \mathbb R^m,
\qquad
\theta \mapsto \bigl(\mathbb E_\theta[f_1], \dots, \mathbb E_\theta[f_m]\bigr).
\]
Thus observable charts can be computed directly from the parameterization.

\subsection{Local observable structure}

Fix a reference distribution $P_0 \in M$, and let $u_0 = \Psi(P_0)$. We are primarily interested in the behavior of $O = \Psi(M)$ in a neighborhood of $u_0$. This local structure captures the distinguishability of nearby distributions. Different choices of observable charts yield different embeddings of the model into Euclidean space. However, when the observables are sufficiently rich, these embeddings are expected to represent the same local structure up to analytic reparameterization. In this sense, observable charts should be viewed as finite-dimensional \emph{presentations} of the local model structure.

\subsection{Discussion}

Observable charts provide a flexible way to represent statistical models using expectation functionals \cite{brown1986fundamentals}. Unlike parameter coordinates, they are directly tied to measurable quantities and naturally factor through the model image $M$. Observable charts should be viewed not merely as representations of a statistical model, but as finite-dimensional presentations of its local structure. When sufficiently rich, they are expected to determine the local model structure up to analytic equivalence.

\medskip

In the following sections, we show that observable charts recover classical statistical geometry at first order, and extend it to capture higher-order structure in singular models.

\section{Observable Tangent Geometry}

\subsection{Setup}

Let $\Phi : \Theta \to \mathcal P(\mathcal X)$ be a parametric statistical model, with $\theta \mapsto P_\theta$, and let $\theta_0 \in \Theta$ be a reference point. For a measurable function $f : \mathcal X \to \mathbb R$ such that $\mathbb E_\theta[f]$ is well-defined in a neighborhood of $\theta_0$, define the observable functional
\[
\psi_f(\theta) = \mathbb E_\theta[f].
\]
Given a finite collection $f_1,\dots,f_m$, we define the observable map
\[
\Psi(\theta) = \big(\psi_{f_1}(\theta), \dots, \psi_{f_m}(\theta)\big).
\]
We study the differential properties of these observable maps at $\theta_0$.

\subsection{Observable derivatives}

Assume that the model admits a density $p_\theta(x)$ with respect to a common dominating measure, and that differentiation under the integral is justified. Let
\[
s_\theta(x) = \nabla_\theta \log p_\theta(x)
\]
denote the score function. For any direction $v \in \mathbb R^d$, the directional derivative of $\psi_f$ at $\theta$ is given by
\[
D\psi_f(\theta)[v]
=
\frac{d}{dt}\Big|_{t=0} \mathbb E_{\theta + t v}[f]
=
\mathbb E_\theta\big[f \cdot v^\top s_\theta\big].
\]
In particular, at $\theta_0$,
\[
D\psi_f(\theta_0)[v]
=
\mathbb E_{\theta_0}\big[f \cdot v^\top s_{\theta_0}\big].
\]
Thus observable derivatives correspond to correlations between $f$ and the score function.

\subsection{Identifiable tangent directions}

A direction $v \in \mathbb R^d$ is said to be \emph{non-identifiable} at $\theta_0$ if
\[
P_{\theta_0 + t v} = P_{\theta_0} + o(t)
\]
in the sense of distributions. Under standard regularity conditions, this is equivalent to
\[
v^\top s_{\theta_0}(x) = 0 \quad \text{almost surely under } P_{\theta_0}.
\]
We define the \emph{identifiable tangent space} as the quotient
\[
T^{\mathrm{id}}_{\theta_0}
=
\mathbb R^d \big/ \{ v : v^\top s_{\theta_0} = 0 \ \text{a.s.} \}.
\]

We identify directions in $T^{\mathrm{id}}_{\theta_0}$ with equivalence classes of parameter perturbations that induce distinct first-order changes in the distribution.
 
\subsection{Observable separation of tangent directions}

We now show that observable derivatives recover the identifiable tangent structure.

\begin{theorem}[Observable Tangent Theorem] \label{thm:observable-tangent}
Let $\mathcal{F} \subset L^2(P_{\theta_0})$ be a class of bounded measurable functions whose linear span is dense in $L^2(P_{\theta_0})$. Then for any direction $v \in \mathbb R^d$,
\[
D\psi_f(\theta_0)[v] = 0 \ \text{for all } f \in \mathcal F
\quad \Longleftrightarrow \quad
v^\top s_{\theta_0} = 0 \ \text{in } L^2(P_{\theta_0}).
\]
In particular, observable derivatives separate identifiable directions, and the observable tangent space coincides with the identifiable tangent space.
\end{theorem}

\begin{proof}[Proof sketch]
From the derivative formula,
\[
D\psi_f(\theta_0)[v]
=
\mathbb E_{\theta_0}[f \cdot g],
\qquad
g = v^\top s_{\theta_0}.
\]
If $D\psi_f(\theta_0)[v]=0$ for all $f \in \mathcal F$, then $\mathbb E[f g]=0$ for all $f$ in a dense subset of $L^2(P_{\theta_0})$, which implies $g=0$ in $L^2(P_{\theta_0})$. The converse is immediate.
\end{proof}

\paragraph{Relation to Fisher geometry} The Fisher information matrix is given by
\[
I(\theta_0) = \mathbb E_{\theta_0}\big[s_{\theta_0} s_{\theta_0}^\top\big].
\]
Observable derivatives can be written as
\[
D\psi_f(\theta_0)[v]
=
\langle f, v^\top s_{\theta_0} \rangle_{L^2(P_{\theta_0})},
\]
so the geometry induced by observable derivatives coincides with the classical Fisher geometry \cite{van2000asymptotic, amari2000methods} on identifiable directions \cite{rothenberg1971identification}.

\subsection{Interpretation}

The preceding result shows that observable derivatives recover the classical tangent-space structure of a statistical model. In particular, first-order observable geometry is equivalent to the geometry induced by the score function and Fisher information. This provides an interpretation of classical statistical theory as a first-order theory of observable functionals. In regular models, this first-order structure is sufficient to describe local behavior. In singular models, however, there exist directions $v$ such that $v^\top s_{\theta_0}=0$, and hence
\[
D\psi_f(\theta_0)[v]=0 \quad \text{for all } f.
\]
Such directions are invisible to all first-order observable derivatives. Understanding these directions requires examining higher-order expansions of observable functionals, which we develop in subsequent sections.

\section{Regular Observable Coordinates}

\subsection{Regular models and identifiable parameterizations}

We consider a point $\theta_0 \in \Theta$ at which the model is regular, in the sense that the score function $s_{\theta_0}(x)$ spans a $d$-dimensional subspace of $L^2(P_{\theta_0})$. Equivalently, the Fisher information matrix
\[
I(\theta_0) = \mathbb E_{\theta_0}\big[s_{\theta_0} s_{\theta_0}^\top\big]
\]
is nonsingular. In this case, the parameterization $\theta \mapsto P_\theta$ is locally identifiable, and the model admits a smooth $d$-dimensional structure near $P_{\theta_0}$.

\subsection{Existence of observable coordinates}

We now show that, at such regular points, the local model structure can be recovered using finitely many observable functionals.

\begin{proposition}[Local observable coordinates at regular points]
Suppose the model is regular at $\theta_0$. Then there exist functions $f_1,\dots,f_d$ such that the observable map
\[
\Psi(\theta) = \big(\mathbb E_\theta[f_1], \dots, \mathbb E_\theta[f_d]\big)
\]
has full-rank Jacobian at $\theta_0$. In particular, $\Psi$ defines a local coordinate system for the model near $\theta_0$.
\end{proposition}

\begin{proof}[Proof sketch]
By Theorem~\ref{thm:observable-tangent}, observable derivatives separate identifiable directions. Since the model is regular, the identifiable tangent space has full dimension $d$. Therefore, there exist functions $f_1,\dots,f_d$ such that the differentials $D\psi_{f_i}(\theta_0)$ are linearly independent. This implies that the Jacobian of $\Psi$ has full rank at $\theta_0$, and the result follows from the inverse function theorem.
\end{proof}

\subsection{Relation to sufficient statistics}

In classical exponential families, observable coordinates arise naturally from sufficient statistics \cite{brown1986fundamentals}. If
\[
p_\theta(x) = \exp\big(\theta^\top T(x) - A(\theta)\big),
\]
then
\[\Psi(\theta) = (
\mathbb E_\theta[T_1(x)],\, \mathbb E_\theta[T_2(x)], \ldots, \mathbb E_\theta[T_d(x)])
\]
provides a global coordinate system for the model. From the observable perspective, sufficient statistics correspond to a distinguished choice of observable chart that is globally valid. The preceding proposition shows that, even outside exponential families, local observable coordinates always exist at regular points.

\subsection{Interpretation}

The existence of local observable coordinates at regular points explains why classical statistical theory can be formulated entirely in terms of first-order quantities such as the score function and Fisher information. In these settings, observable functionals provide a complete first-order description of the model. In particular, the parameterization $\theta$ itself can be viewed as one possible coordinate system, while observable charts provide alternative coordinate systems derived from expectation functionals.

\subsection{Limitations in singular models}

The situation changes fundamentally at singular points. When the Fisher information is degenerate, the identifiable tangent space has reduced dimension, and observable derivatives fail to distinguish certain directions. In such cases, no observable chart can provide a full first-order coordinate system. This motivates the study of higher-order observable expansions, which allow hidden directions to be detected at higher orders. We develop this perspective in the next section.

\section{Observable Order and KL Geometry}

\subsection{Observable order along arcs}

Let $\theta_0 \in \Theta$ be a reference point, and let $\gamma : (-\epsilon,\epsilon) \to \Theta$ be an analytic curve such that $\gamma(0)=\theta_0$. Given an observable chart $\Psi$, we define the \emph{observable order} of $\gamma$ as
\[
o_\Psi(\gamma)
=
\operatorname{ord}_{t=0}\big(\Psi(\gamma(t))-\Psi(\theta_0)\big),
\]
where $\operatorname{ord}_{t=0}$ denotes the smallest integer $k$ such that the expansion has a nonzero $t^k$ term,
\[
\Psi(\gamma(t))-\Psi(\theta_0)=O(t^k)
\quad\text{but}\quad
\Psi(\gamma(t))-\Psi(\theta_0)\neq O(t^{k+1}).
\]
If no such $k$ exists, define $o_\Psi(\gamma)=\infty$. Equivalently, $o_\Psi(\gamma)=k$ if
\[
\Psi(\gamma(t))-\Psi(\theta_0)=t^k v + o(t^k)
\]
for some nonzero vector $v\in\mathbb R^m$.  If $o_\Psi(\gamma)=1$, the curve is detectable at first order and corresponds to a direction visible in the observable tangent space. If $o_\Psi(\gamma) > 1$, then the curve is invisible to all first-order observable derivatives and only becomes distinguishable at higher order. Thus observable order provides a natural notion of \emph{higher-order identifiability}.

\subsection{KL divergence along arcs}

Let $K(\theta) = KL(P_{\theta_0} \| P_\theta)$ denote the Kullback--Leibler divergence from the reference distribution. We define the \emph{KL order} of $\gamma$ as $o_K(\gamma)= \operatorname{ord}_{t=0} K(\gamma(t))$. In regular models, $K(\theta)$ admits a quadratic expansion near $\theta_0$ \cite{van2000asymptotic},
\[
K(\theta_0 + v) \approx \tfrac{1}{2} v^\top I(\theta_0) v,
\]
so for a first-order identifiable direction, $o_K(\gamma)=2$.

\subsection{Principle of observable completeness}
We now relate observable order to the behavior of the KL divergence. 

\begin{definition}[First-order complete observable chart] \label{def:first-order-complete}
An observable chart $\Psi$ is \textit{first-order complete} at $\theta_0$ if for any analytic curve $\gamma$ with $\gamma(0)=\theta_0$,
\[
\Psi(\gamma(t)) - \Psi(\theta_0) = o(t)
\quad\Longrightarrow\quad
p_{\gamma(t)} - p_{\theta_0} = o(t)
\]
in $L^2(P_{\theta_0})$.
\end{definition}

Equivalently, observable chart $\Psi$ is first-order complete at $\theta_0$ if 
\[
\ker D\Psi(\theta_0)
=
\{v \in \mathbb{R}^d : v^\top s_{\theta_0} = 0 \text{ in } L^2(P_{\theta_0})\},
\]
 i.e., $\Psi$ induces an injective map on the identifiable tangent space $T^{\mathrm{id}}_{\theta_0}$. These formulations are equivalent under differentiability of the model, as the first-order behavior along curves is determined by the score function. Thus observable order is an intrinsic property of the model and does not depend on the particular choice of complete observable chart.

\begin{remark}
By Theorem~\ref{thm:observable-tangent}, first-order completeness holds whenever the observable derivatives separate identifiable score directions. In practice, completeness is assessed by examining observable orders of relevant perturbations and verifying that observable derivatives separate identifiable directions.
\end{remark}

In singular models, however, certain directions are invisible to first-order derivatives. This motivates a higher-order notion of completeness.

\begin{definition}[$k$-th order observable completeness] 
An observable chart $\Psi$ is $k$-th order complete at $\theta_0$ if for any analytic curve $\gamma$ with $\gamma(0)=\theta_0$,
\[
\Psi(\gamma(t)) - \Psi(\theta_0) = o(t^k)
\quad\Longrightarrow\quad
p_{\gamma(t)} - p_{\theta_0} = o(t^k)
\]
in $L^2(P_{\theta_0})$.
\end{definition}
First-order completeness corresponds to the case $k=1$. Higher-order completeness ensures that observable charts detect structure up to a prescribed order. 

\begin{remark}
An equivalent formulation of $k$-th order completeness is that for any analytic curves $\gamma_1,\gamma_2$,
\[
\Psi(\gamma_1(t)) - \Psi(\gamma_2(t)) = o(t^k)
\quad\Longrightarrow\quad
p_{\gamma_1(t)} - p_{\gamma_2(t)} = o(t^k).
\]
This expresses that observable charts distinguish distributions up to order $k$.
\end{remark}

\begin{remark}
This notion can be interpreted in terms of $k$-jets of curves, where observable completeness corresponds to injectivity of observable charts on $k$-th order expansions.
\end{remark}

\subsection{Relation between observable order and KL order}
We assume the model admits a second-order expansion of KL divergence along $L^2$-small perturbations, which holds under standard smoothness and domination conditions (e.g., differentiable densities with common support). The following result shows that observable order provides a lower bound on the rate of KL divergence decay, linking observable geometry directly to statistical distinguishability.

\begin{theorem}[Observable order controls KL order] \label{thm:observable-kl}
Let $\Psi$ be a first-order complete observable chart at $\theta_0$. Then for any analytic curve $\gamma$,
\[
o_K(\gamma) \ge 2\,o_\Psi(\gamma).
\]
\end{theorem}

Thus observable order provides a lower bound on the rate at which statistical distinguishability emerges along analytic paths. This result shows that observable order provides an intrinsic notion of statistical distinguishability. In regular models, this reduces to the classical quadratic behavior governed by Fisher information, while in singular models it captures higher-order degeneracies invisible to first-order analysis.
\begin{proof}
Let $k = o_\Psi(\gamma)$, so that
\[
\Psi(\gamma(t)) - \Psi(\theta_0) = O(t^k).
\]
By first-order completeness,
\[
p_{\gamma(t)} - p_{\theta_0} = O(t^k)
\]
in $L^2(P_{\theta_0})$.

Using the standard second-order expansion of KL divergence \cite{watanabe2009algebraic},
\[
K(\gamma(t)) = \frac{1}{2}\|p_{\gamma(t)} - p_{\theta_0}\|^2 + o(\|p_{\gamma(t)} - p_{\theta_0}\|^2),
\]
we obtain
\[
K(\gamma(t)) = O(t^{2k}),
\]
which implies $o_K(\gamma) \ge 2o_\Psi(\gamma)$.
\end{proof}

\begin{remark}[Generic equality]
If the leading quadratic form in the KL expansion is nondegenerate along the first non-vanishing observable displacement, then
\[
o_K(\gamma) = 2\,o_\Psi(\gamma).
\]
This condition holds in many models of interest, including Gaussian regression, mixture models, and neural networks, as illustrated in Section \ref{sec:examples}.
\end{remark}

\subsection{Interpretation}

The preceding result shows that observable order provides a lower bound on the rate at which statistical distinguishability emerges along analytic paths. In particular, directions that are invisible at first order ($o_\Psi(\gamma) > 1$) correspond to higher-order degeneracies of the KL divergence. This provides a geometric explanation for the failure of classical asymptotic theory in singular models: Fisher information captures only first-order observable structure, while singular behavior arises from directions that become observable only at higher order \cite{watanabe2009algebraic}.

\subsection{Towards higher-order observable geometry}

Observable order suggests a hierarchy of geometric structure:
\begin{itemize}
    \item First-order observables recover classical tangent-space geometry.
    \item Higher-order observable expansions detect directions invisible at first order.
    \item This suggests that the local geometry of a statistical model may be recoverable from observable jet data along analytic arcs.
\end{itemize}
This perspective extends classical differential geometry of statistical models to a higher-order setting, in which distinguishability is determined not only by first derivatives but by the full hierarchy of observable expansions.

\section{Constructing Observable Charts}

\subsection{The construction problem}

The preceding sections establish that observable charts can recover both first-order and higher-order structure of a statistical model. A natural question is how to construct such charts in practice. In general, there is no canonical choice of observables. Instead, observable charts must be chosen so as to capture the local structure of the model near a reference point $P_{\theta_0}$. The goal is to construct a finite collection of observables that provides a sufficiently rich local presentation of the model.

\subsection{Iterative construction procedure}

We now describe a practical procedure for constructing observable charts.

\begin{enumerate}[leftmargin=1.5em]
    \item \textbf{Initialize with natural observables.}  
    Choose a small collection of observables based on the structure of the model. Typical choices include moments, cross-moments, or response functionals \cite{anandkumar2014tensor}.
    
    \item \textbf{Compute first-order structure.}  
    Evaluate the Jacobian $D\Psi(\theta_0)$ and identify directions $v$ such that
    \[
    D\Psi(\theta_0)[v] = 0.
    \]
    These directions are invisible to first-order observables.
    
    \item \textbf{Probe hidden directions.}  
    For each such direction, construct analytic curves $\gamma(t)$ with initial direction $v$, and compute observable expansions
    \[
    \Psi(\gamma(t)) - \Psi(\theta_0).
    \]
    
    \item \textbf{Add higher-order observables.}  
    Introduce additional observables whose expectations become nonzero at the lowest possible order along these curves.
    
    \item \textbf{Iterate.}  
    Repeat the process until all directions of interest become observable at some finite order.
\end{enumerate}

\subsection{Interpretation}

This procedure constructs observable charts by progressively revealing directions that are hidden from lower-order observables. At each stage, new observables are added specifically to detect the lowest-order nonvanishing effects along previously invisible directions. In this sense, observable charts are built to minimize the order at which different directions become distinguishable.

\subsection{Choice of observables}

In practice, effective observable choices depend on the model class.
\begin{itemize}
    \item In mixture models, low-order moments and cumulants naturally reveal component structure.
    \item In regression and matrix factorization models, cross-moments provide intrinsic coordinates.
    \item In neural networks, response functionals or evaluations at input points capture functional variation.
\end{itemize}
These choices are not arbitrary: they are guided by the requirement that observable expansions detect hidden structure at the lowest possible order.

\subsection{Discussion}

The construction procedure does not require identifying a minimal or canonical set of observables. Rather, it provides a principled method for building finite observable charts that capture the local structure of a model to a desired level of precision. In particular, low-order observable completeness is often sufficient for understanding singular behavior, as higher-order effects become progressively smaller. This perspective shifts the focus from parameter-based descriptions to observable-based constructions, providing a practical framework for analyzing singular statistical models.

\section{Examples} \label{sec:examples}

We illustrate the observable framework in several canonical models. In each case, we construct observable charts that recover the local structure of the model and demonstrate how higher-order observables reveal directions that are invisible at first order. In these examples, observable order aligns with known singular behavior characterized by learning coefficients, suggesting a deeper connection between observable geometry and asymptotic learning rates.

\subsection{Gaussian mixture model}

Consider a two-component Gaussian mixture with fixed variance:
\[
P = \Big(\tfrac{1}{2} + \alpha\Big)\mathcal N(\mu + \delta, \sigma^2)
+
\Big(\tfrac{1}{2} - \alpha\Big)\mathcal N(\mu - \delta, \sigma^2).
\]
We study the model near the singular point $(\mu,\delta,\alpha) = (\mu_0, 0, 0)$.

\paragraph{Step 1: mean.}
The first observable is the mean:
\[
m_1 = \mathbb E[X] = \mu + 2\alpha \delta.
\]
At the singular point, the derivative of $m_1$ detects only the $\mu$ direction; the parameters $\delta$ and $\alpha$ are invisible at first order.

\paragraph{Step 2: variance.}
The variance is
\[
\kappa_2 = \operatorname{Var}(X) = \sigma^2 + \delta^2 + O(\alpha^2 \delta^2).
\]
This reveals the separation parameter $\delta$ at second order.

\paragraph{Step 3: skewness.}
The third cumulant satisfies
\[
\kappa_3 \approx 4 \alpha \delta^3
\]
near the singular point. This observable detects asymmetry through the interaction between $\alpha$ and $\delta$ \cite{rothenberg1971identification}.

\paragraph{Observable chart.}
The collection
\[
\Psi = (m_1, \kappa_2, \kappa_3)
\]
provides a finite observable chart that captures the local singular structure. In particular:
\begin{itemize}
    \item $\mu$ is observable at order $1$,
    \item $\delta$ is observable at order $2$,
    \item $\alpha$ appears through mixed terms at order $3$.
\end{itemize}
This example illustrates how higher-order observables are required to recover directions that are invisible at first order.

\subsection{One-unit neural network}

Consider the model
\[
f_\theta(x) = a \, \tanh(wx + b),
\]
with parameters $\theta = (a,w,b)$. We study the inactive-unit singularity at $(a,w,b) = (0,w_0,b_0)$ \cite{watanabe2009algebraic}.

\paragraph{First-order observables.}
For a test function $\phi$, define
\[
\psi_\phi(\theta) = \mathbb E[f_\theta(X)\phi(X)].
\]
Then
\[
\psi_\phi(\theta) = a \, G_\phi(w,b),
\]
where $G_\phi(w,b) = \mathbb E[\tanh(wX+b)\phi(X)]$. At $a=0$, the derivative with respect to $w$ and $b$ vanishes, so these directions are invisible at first order.

\paragraph{Higher-order structure.}
Expanding $G_\phi(w,b)$ near $(w_0,b_0)$ yields
\[
\psi_\phi(\theta)
=
a G_\phi(w_0,b_0)
+
a(w-w_0)\partial_w G_\phi(w_0,b_0)
+
a(b-b_0)\partial_b G_\phi(w_0,b_0)
+ \cdots
\]
Thus the hidden directions $w$ and $b$ appear through mixed second-order terms.

\paragraph{Observable chart.}
By choosing sufficiently many test functions $\phi_1,\phi_2,\phi_3$, one obtains an observable chart of the form
\[
\Psi(\theta) \approx \big(a,\; a(w-w_0),\; a(b-b_0)\big),
\]
which captures the local singular structure.

\subsection{Reduced rank regression}

Consider the reduced rank regression model
\[
Y = BX + \varepsilon,
\]
where $B \in \mathbb R^{p \times q}$ has rank at most $r$. The model can be parameterized as $B = UV^\top$ with $U \in \mathbb R^{p \times r}$ and $V \in \mathbb R^{q \times r}$. A natural observable is the cross-moment
\[
\Psi(P) = \mathbb E[YX^\top].
\]
Under the model, this observable recovers $B$ exactly:
\[
\mathbb E[YX^\top] = B \, \mathbb E[XX^\top].
\]
Assuming $\mathbb E[XX^\top]$ is nonsingular, $\Psi$ provides a coordinate system for the model in terms of $B$. In particular, the observable chart directly parameterizes the intrinsic model space without redundancy. At regular points (where $\operatorname{rank}(B)=r$), this observable provides a full first-order description of the model. 

Singularities arise at rank-deficient points, where the tangent space degenerates. In these cases, the observable chart still provides a natural representation of the intrinsic geometry, even though the parameterization $(U,V)$ is non-identifiable. 

\subsubsection{Reduced rank regression: a $2 \times 2$ rank-1 example}
We consider a simple reduced rank regression model
\[
Y = BX + \varepsilon,
\]
where $B \in \mathbb R^{2 \times 2}$ has rank at most $1$. A convenient parameterization is
\[
B = uv^\top,
\quad
u \in \mathbb R^2,\; v \in \mathbb R^2.
\]

Assume that $X$ is centered with covariance $\mathbb E[XX^\top] = I_2$, and that $\varepsilon$ is independent noise with zero mean.

\paragraph{Observable coordinates.}
A natural observable is the cross-moment
\[
\Psi(P) = \mathbb E[YX^\top].
\]
Under the model,
\[
\mathbb E[YX^\top] = B,
\]
so the observable coordinates are the entries
\[
m_{ij} = \mathbb E[Y_i X_j], \qquad i,j=1,2.
\]
Thus the observable map identifies the model with a subset of $\mathbb R^4$.

\paragraph{Rank constraint.}
The rank-1 structure imposes the relation
\[
m_{11}m_{22} - m_{12}m_{21} = 0.
\]
Hence the observable image is the determinantal variety
\[
O = \left\{ m \in \mathbb R^4 : m_{11}m_{22} - m_{12}m_{21} = 0 \right\}.
\]

\paragraph{Regular and singular points.}
At points where $B \neq 0$ has rank $1$, this variety is locally a smooth $3$-dimensional manifold, and the observable coordinates provide a regular local parameterization. At the singular point $B=0$, however, the gradient of the defining equation vanishes, and the first-order tangent space coincides with the full ambient space $\mathbb R^4$. Thus the rank constraint is invisible to first-order observable derivatives.

\paragraph{Second-order structure.}
The rank constraint appears only at second order through the quadratic relation
\[
\Delta = m_{11}m_{22} - m_{12}m_{21}.
\]
This shows that the singular structure is encoded in higher-order observable relations.

\paragraph{Observable order along arcs.}
Consider an analytic curve through the singular point defined by
\[
u(t) = t a, \qquad v(t) = t b,
\]
so that
\[
B(t) = u(t)v(t)^\top = t^2 ab^\top.
\]
Then each observable coordinate satisfies
\[
m_{ij}(t) = O(t^2),
\]
so
\[
o_\Psi(\gamma) = 2.
\]
In Gaussian regression models, the Kullback--Leibler divergence satisfies
\[
K(B) \asymp \|B\|^2,
\]
and therefore along this curve
\[
K(\gamma(t)) = O(t^4),
\]
so that
\[
o_K(\gamma) = 4 = 2\,o_\Psi(\gamma).
\]
This example illustrates that observable order provides an intrinsic description of singular structure, independent of parameterization.

\paragraph{Interpretation.}
This example illustrates that the singularity at $B=0$ is invisible to first-order observable geometry and first appears at second order. The observable chart detects this through quadratic relations among observable coordinates, providing a concrete instance of higher-order observable structure. This provides a concrete verification of Theorem~\ref{thm:observable-kl}, showing that observable order correctly predicts the rate of KL decay in a singular model. Thus, although the rank constraint is completely invisible to first-order observable geometry, it is recovered exactly at second order through observable relations.

\subsection{Summary}

Across these examples, the observable construction follows a common pattern:
\begin{itemize}
    \item First-order observables recover classical tangent directions.
    \item Certain directions are invisible at first order.
    \item Additional observables reveal these directions at higher order.
\end{itemize}
This demonstrates that singular structure can be understood in terms of the order at which different directions become observable.

\section{Discussion}

\subsection{Summary of contributions}

We introduced an observable-based framework for studying statistical models, in which model structure is presented through expectation functionals rather than parameterizations. Observable charts provide finite-dimensional presentations of the model image and can be used to analyze both regular and singular behavior. We showed that:
\begin{itemize}
    \item Observable derivatives recover the classical identifiable tangent space, providing a first-order description equivalent to Fisher geometry.
    \item In regular models, finitely many observables provide local coordinate systems, explaining why classical asymptotic theory depends only on first-order quantities.
    \item In singular models, certain directions are invisible to first-order observables and become detectable only through higher-order observable expansions.
    \item Observable order along analytic curves provides a natural way to characterize higher-order identifiability and is closely tied to the behavior of the Kullback--Leibler divergence.
\end{itemize}
Together, these results suggest that classical statistical theory can be viewed as a first-order approximation within a broader observable geometry. In particular, observable geometry provides a coordinate-free description of local model structure that separates intrinsic statistical behavior from parameterization artifacts.

\subsection{Relation to singular learning theory}

Singular learning theory (SLT) characterizes asymptotic behavior in terms of the geometry of the Kullback--Leibler divergence near singularities, typically using resolution of singularities and algebraic techniques \cite{watanabe2009algebraic, watanabe2010asymptotic}. The observable framework developed here provides a complementary perspective. Rather than working in parameter space and resolving singularities through geometric transformations, we describe singular structure directly in terms of observable quantities and their expansion orders. In particular, observable order along analytic arcs parallels the valuation structure used in defining RLCT, suggesting a potential intrinsic reformulation of these invariants.

\subsection{Observable charts as local presentations}

Observable charts should be viewed not merely as representations of a model, but as local presentations of its structure. Observable coordinates are expected to determine the local behavior of the model up to analytic equivalence in sufficiently rich settings. This perspective emphasizes that the parameter space $\Theta$ serves only as a convenient chart for constructing models, while the intrinsic object of interest is the model image $M$. The quadratic relationship between KL divergence and observable displacement is not an assumption but a consequence of first-order completeness together with the second-order expansion of KL divergence. Observable charts provide alternative coordinate systems that factor through $M$ and avoid redundancies arising from non-identifiability.

\subsection{Practical implications}

The observable construction procedure provides a principled way to build finite-dimensional summaries of complex models that capture relevant structure to a desired order. In practice, low-order observable completeness is often sufficient to reveal the dominant singular behavior. This perspective may be useful for:
\begin{itemize}
    \item designing diagnostics for singularity and non-identifiability,
    \item constructing reduced representations of models,
    \item guiding approximation methods that depend on local geometry.
\end{itemize}

\subsection{Limitations}

The present work focuses on local structure near a fixed reference distribution and does not address global properties of the model. In particular:
\begin{itemize}
    \item We do not provide canonical or minimal sets of observables.
    \item While observable completeness is formally defined, a full intrinsic characterization independent of parameterization remains open.
    \item Connections to existing invariants in singular learning theory are not yet formalized.
\end{itemize}
In addition, while observable order captures higher-order structure, computing it explicitly may be challenging in complex models.

\subsection{Future directions}

Several directions for future work arise naturally from this framework. First, it would be valuable to formalize the relationship between observable order and invariants such as the real log canonical threshold (RLCT), potentially providing new tools for analyzing singular asymptotics. Second, the observable framework suggests a hierarchy of geometric structure based on higher-order expansions. Developing a systematic theory of higher-order observable geometry may provide a unified description of singular models. Third, extending these ideas to Bayesian settings, including posterior behavior and model comparison criteria, may yield new insights into existing methods such as WAIC and WBIC. Finally, further work is needed to develop computational methods for constructing observable charts in high-dimensional and nonparametric models.

\subsection{Conclusion}
We have proposed an observable-based perspective on statistical models that unifies classical and singular settings. By focusing on expectation functionals and their expansions, this framework provides a direct and intrinsic way to study identifiability and singular structure, complementing existing approaches based on parameter space geometry.

\section*{AI Acknowledgment}
The author used a large language model to assist with editing and organization. All mathematical content and results were independently developed and verified by the author.

\bibliographystyle{plain}
\bibliography{refs}
\end{document}